\documentclass[english]{smfart}
\usepackage{amsmath, amssymb, amsthm, amscd,smfenum,xspace,mathrsfs,url,upref,verbatim}
\usepackage[utf8]{inputenc} 
\usepackage[T1]{fontenc}
\usepackage{dsfont}

\usepackage{mathabx}
\usepackage{stmaryrd}

\usepackage[all,cmtip]{xy}
\usepackage{xfrac} 

\usepackage{ textcomp }

\usepackage{tikz-cd}
\usepackage{tikz}

\let\hat=\widehat
\let\tilde=\widetilde
\numberwithin{equation}{subsection}
\newtheorem{conjecture}{Conjecture} 
\newtheorem{theorem}{Theorem} 
\newtheorem{question}{Question} 
\newtheorem{proposition}[equation]{Proposition}
\newtheorem{lemme}[equation]{Lemma}

\theoremstyle{remark}

\DeclareMathOperator{\DR}{DR}
\DeclareMathOperator{\reg}{reg}

\DeclareMathOperator{\ord}{ord}

\DeclareMathOperator{\id}{id}

\def\cartesien{\ar@{}[rd]|{\square}}

\DeclareMathOperator{\car}{car}

\DeclareMathOperator{\Supp}{Supp}
\DeclareMathOperator{\divi}{div}
\DeclareMathOperator{\Sing}{Sing}
\DeclareMathOperator{\Reg}{Reg}
\DeclareMathOperator{\hol}{hol}
\DeclareMathOperator{\El}{El}
\DeclareMathOperator{\rg}{rg}
\DeclareMathOperator{\Sl}{Sl}
\DeclareMathOperator{\Irr}{Irr}
\DeclareMathOperator{\Sol}{Sol}
\DeclareMathOperator{\dR}{dR}
\DeclareMathOperator{\an}{an}

\DeclareMathOperator{\rd}{rd}
\DeclareMathOperator{\et}{\text{ét}}
\DeclareMathOperator{\nb}{nb}
\usepackage{enumerate}

\author[J.-B.~Teyssier]{Jean-Baptiste Teyssier}
\curraddr{Freie Universität Berlin, Mathematisches Institut, Arnimallee 3, 14195 Berlin, Germany}
\email{teyssier@zedat.fu-berlin.de}
\title{A Boundedness theorem for nearby slopes of holonomic $\mathcal{D}$-modules}

\usepackage[english]{babel}
\begin{document}

\begin{abstract}
Using twisted nearby cycles, we define a new notion of slopes for complex holonomic $\mathcal{D}$-modules. We prove a boundedness result for these slopes, study their functoriality and use them to characterize regularity. For a family of (possibly irregular) algebraic connections $\mathcal{E}_t$ parametrized by a smooth curve, we deduce under natural conditions an explicit bound for the usual slopes of the differential equation satisfied by the family of irregular periods of the $\mathcal{E}_t$. This generalizes the regularity of the Gauss-Manin connection proved by Katz and Deligne. Finally, we address some questions about analogues of the above results for wild ramification in the arithmetic context.
\end{abstract}
\maketitle

Let $V$ be a smooth algebraic variety over a finite field of characteristic $p>0$, and let $U$ be an open subset in $V$ such that $D:=V\setminus U$ is a normal crossing divisor. Let $\ell$ be a prime number different from $p$. Using restriction to curves, Deligne defined  \cite{Delignebornee}  a notion of $\ell$-adic local system on $U$ with bounded ramification along $D$. Such a definition is problematic to treat functoriality questions: the direct image of a local system is not a local system any more, duality does not commute with restriction in general. In this paper, we investigate the characteristic 0 aspect of this problem, that is the 
\begin{question}
Let $X$ be a complex manifold. Can one define a notion of holonomic $\mathcal{D}_X$-module with bounded irregularity which has good functoriality properties?
\end{question}
In dimension 1, to bound the irregularity number of a $\mathcal{D}$-module with given generic rank amounts to bound its slopes. Let $\mathcal{M}$  be a holonomic $\mathcal{D}_X$-module and let $Z$ be a hypersurface of $X$. Mebkhout \cite{Mehbgro}  showed that the \textit{irregularity complex} $\Irr_{Z}(\mathcal{M})$  of $\mathcal{M}$  along $Z$ is a perverse sheaf endowed with a $\mathds{R}_{>1}$ increasing locally finite filtration by sub-perverse sheaves $\Irr_{Z}(\mathcal{M})(r)$. In dimension one, this construction gives back the usual notion of slope modulo the change of variable $r\longrightarrow 1/(r-1)$. The \textit{analytic slopes of $\mathcal{M}$ along $Z$} are the $r>1$ for which the supports of the graded pieces of $(\Irr_{Z}(\mathcal{M})(r))_{r>1}$ are non empty. \\ \indent
The existence of a uniform bound  in  $Z$ is not clear a priori. We thus formulate the following
\begin{conjecture}\label{conj0}
Locally on $X$, the set of analytic slopes of a holonomic 
$\mathcal{D}_X$-module is bounded.
\end{conjecture}
This statement means that for a holonomic $\mathcal{D}_X$-module $\mathcal{M}$, one can find for every point in  $X$ a neighbourhood
$U$ and a constant $C>0$ such that the analytic slopes of $\mathcal{M}$ along any germ of hypersurface in $U$ are $\leq C$.\\ \indent
On the other hand, Laurent defined \textit{algebraic slopes} using his theory of micro-characteristic varieties \cite{ LaurentPolygone}. From Laurent and Mebkhout work \cite{LM}, we know that the set of analytic slopes of a holonomic  $\mathcal{D}$-module $\mathcal{M}$ along $Z $ is equal to the set of algebraic slopes of $\mathcal{M}$ along $Z$. Since micro-characteristic varieties are invariant by duality, we deduce that analytic slopes are invariant by duality.\\ \indent

The aim of this paper is to define a third notion of slopes and to investigate some of its properties. The main idea lies in the observation that for a germ $\mathcal{M}$ of $\mathcal{D}_{\mathds{C}}$-module at $0\in \mathds{C}$, \textit{the slopes of  $\mathcal{M}$ at 0 are encoded in the vanishing of certain nearby cycles}. We show in \ref{propdim1} that $r\in \mathds{Q}_{\geq 0}$ is a slope for $\mathcal{M}$ at 0 if and only if one can find a germ $N$ of meromorphic connection at 0 with slope $r$ such that $\psi_{0}(\mathcal{M}\otimes N)\neq 0$.\\ \indent
We thus introduce the following definition. Let $X$ be a complex manifold and let $\mathcal{M}$ be an object of the derived category $\mathcal{D}_{\hol}^b(X)$ of complexes of  $\mathcal{D}_X$-modules with bounded and holonomic cohomology. Let $f\in \mathcal{O}_X$. We denote by $\psi_f$ the nearby cycle functor\footnote{For general references on the nearby cycle functor, let us mention \cite{Kashpsi},\cite{Malpsi},\cite{MS} and \cite{MM}.} associated to $f$. We define the \textit{nearby slopes of $\mathcal{M}$ associated to $f$} to be the set $\Sl_{f}^{\nb}(\mathcal{M})$ complement in $\mathds{Q}_{\geq 0}$ of the set of rationals $r\geq 0$ such that for every germ $N$ of meromorphic connection  at 0 with slope $r$, we have 
\begin{equation}\label{annulation0}
\psi_{f}(\mathcal{M}\otimes f^{+}N)\simeq 0
\end{equation}
Let us observe that the left-hand side of \eqref{annulation0} depends on $N$ only via $\hat{\mathcal{O}}_{\mathds{C},0}\otimes_{\mathcal{O}_{\mathds{C},0}} N$, and that nearby slopes are sensitive to the reduced structure of $\divi f$, whereas the analytic and algebraic slopes only see the support of $\divi f$. \\ \indent
Twisted nearby cycles appear for the first time in the algebraic context in \cite{DeligneLettreMalgrange}. Deligne proves in \textit{loc. it.}  that for a given function $f$, the set of $\hat{\mathcal{O}}_{\mathds{C},0}$-differential modules $N$ such that 
$\psi_{f}(\mathcal{M}\otimes f^{+}N)\neq 0$ is finite.\\ \indent
The main result of this paper is an affirmative answer to  conjecture \ref{conj0} for nearby slopes, that is the 
\begin{theorem}\label{theoremprincipal}
Locally on $X$, the set of nearby slopes of a holonomic 
$\mathcal{D}$-module is bounded.
\end{theorem}
This statement means that for a holonomic $\mathcal{D}_X$-module $\mathcal{M}$, one can find for every point in  $X$ a neighbourhood
$U$ and a constant $C>0$ such that the nearby slopes of  $\mathcal{M}$ associated to any $f\in \mathcal{O}_U$ are $\leq C$.\\ \indent
For meromorphic connections with good formal structure, we show the following refinement
\begin{theorem}\label{theoremprincipalraffiné}
Let $\mathcal{M}$ be a meromorphic connection with good formal structure. Let $D$ be the pole locus of $\mathcal{M}$ and let $D_1, \dots, D_n$ be the irreducible components of $D$. We denote by $r_{i}(\mathcal{M})\in \mathds{Q}_{\geq 0}$
the highest generic slope of $\mathcal{M}$ along $D_i$. Then, the nearby slopes of $\mathcal{M}$ are  $\leq r_{1}(\mathcal{M})+\cdots +r_{n}(\mathcal{M})$.
\end{theorem}
The main tool used in the proof of theorem \ref{theoremprincipal}  is a structure theorem for formal meromorphic connections first conjectured in \cite{CastroSabbah}, studied by Sabbah \cite{Sabbahdim} and proved by Kedlaya \cite{Kedlaya1}\cite{Kedlaya2} in the context of excellent schemes and analytic spaces, and independently by  Mochizuki \cite{Mochizuki2}\cite{Mochizuki1} in the algebraic context.\\ \indent
Let us give some details on the strategy of the proof of theorem \ref{theoremprincipal}. A dévissage carried out in  \ref{reduction} allows one to suppose that $\mathcal{M}$  is a meromorphic connection. Using Kedlaya-Mochizuki theorem, one reduces further the proof to the case where 
 $\mathcal{M}$ has good formal structure. We are thus left to prove theorem \ref{theoremprincipalraffiné}. We resolve the singularities of $Z$. The problem that occurs at this step is that a randomly chosen embedded resolution $p:\tilde{X}\longrightarrow X$ will increase the generic slopes of $\mathcal{M}$ in a way that cannot be controlled. We show in \ref{propositionprincipale}  that a fine version of embedded resolution \cite{BMUniformization} allows  to control the generic slopes of $p^{+}\mathcal{M}$  in terms of the sum $ r_{1}(\mathcal{M})+\cdots +r_{n}(\mathcal{M})$ and  the multiplicities of $p^{\ast}Z$. A crucial tool for this is a theorem
 \cite[I 2.4.3]{Sabbahdim} proved by Sabbah in dimension 2 and by Mochizuki \cite[2.19]{MochStokes} in any dimension relating the good formal models appearing at a given point  with the generic models on the divisor locus. Using a vanishing criterion \ref{lemmeannulation}, one finally proves \eqref{annulation0} for $r> r_{1}(\mathcal{M})+\cdots +r_{n}(\mathcal{M})$.\\ \indent
Let $\mathcal{M}\in \mathcal{D}_{\hol}^b(X)$ and let us denote by $\mathds{D}\mathcal{M}$ the dual complex of $\mathcal{M}$. The nearby slopes satisfy the following functorialities
\begin{theorem}\label{theoremprincipal-1}
$(i)$ For every $f\in \mathcal{O}_X$, we have $$\Sl_{f}^{\nb}(\mathds{D}\mathcal{M})=\Sl_{f}^{\nb}(\mathcal{M})$$
$(ii)$ Let $p:X\longrightarrow Y$ be a proper morphism and let $f\in \mathcal{O}_Y$ such that $p(X) $ is not contained in $f^{-1}(0)$. Then
$$\Sl_{f}^{\nb}(p_{+}\mathcal{M})\subset \Sl_{fp}^{\nb}(\mathcal{M})$$
\end{theorem}
Let us observe that  $(ii)$  is a direct application of the compatibility of nearby cycles with proper direct image \cite{MS}.\\ \indent
It is an interesting problem to try to compare nearby slopes and analytic slopes. This question won't be discussed in this paper, but we characterize regular holonomic $\mathcal{D}$-modules using nearby slopes.
\begin{theorem}\label{comparaisonreg}
A complex $\mathcal{M}\in \mathcal{D}_{\hol}^b(X)$ is regular if and only if for every quasi-finite morphism $\rho: Y\longrightarrow X$ with $Y$ a complex manifold, the set of nearby slopes of $\rho^{+}\mathcal{M}$ is contained in $\{ 0\}$.
\end{theorem}
For an other characterization  of regularity (harder to deal with in practice) using  $R\mathcal{H}om$ and the solution functor, we refer to \cite{carmodreg}.\\ \indent
Let us give an application of the preceding results. Let $U$ be a smooth complex algebraic variety and let $\mathcal{E}$  be an algebraic connection on $U$. We denote by $H^{k}_{\dR}(U,\mathcal{E})$ the $k^{th}$ de Rham cohomology group of  $\mathcal{E}$, and by $\mathcal{V}$ the local system of horizontal sections of $\mathcal{E}^{\an}$ on  $U^{\an}$. If $\mathcal{E}$  is regular, Deligne proved  \cite{Del} that the canonical comparison morphism
\begin{equation}\label{accperiode}
H^{k}_{\dR}(U,\mathcal{E} )\longrightarrow H^{k}(U^{\an},\mathcal{V} )
\end{equation}
is an isomorphism. If $\mathcal{E}$  is the trivial connection, this is due to Grothendieck \cite{GroDR}. In the irregular case, \eqref{accperiode}  is no longer an isomorphism. It can happen that $H^{k}_{\dR}(U,\mathcal{E} )$ is non zero and $H^{k}(U^{\an},\mathcal{V})$  is zero, which means that there are not enough topological cycles in $U^{\an}$. The \textit{rapid decay homology} $H_{k}^{\rd}(U,\mathcal{E}^{\ast} )$ needed  to remedy this problem appears in dimension one in \cite{HB} and in higher dimension in \cite{Hiendimdeux}\cite{HienInv}. It includes cycles drawn on a compactification of $U^{\an}$ taking into account the asymptotic at infinity of the solutions of the dual connection  $\mathcal{E}^{\ast}$. By Hien duality theorem, we have a perfect pairing
\begin{equation}\label{accperiodeirr}
\int : H^{k}_{\dR}(U,\mathcal{E} )\times H_{k}^{\rd}(U,\mathcal{E}^{\ast} )\longrightarrow \mathds{C}
\end{equation}
For $\omega\in H^{k}_{\dR}(U,\mathcal{E})$ and $\gamma\in H_{k}^{\rd}(U,\mathcal{E}^{\ast})$, the complex number  $\int_\gamma \omega$ is a \textit{period for $\mathcal{E}$}.\\ \indent
Let $f:X\longrightarrow S$ be a proper and generically smooth morphism, where $X$ denotes an algebraic variety and $S$ denotes a neighbourhood of 0 in $\mathds{A}^{1}_{\mathds{C}}$. Let $U$ be the complement of a normal crossing divisor $D$ of $X$ such that for every $t\neq 0$ close enough to 0, $D_t$ is a normal crossing divisor of  $X_t$. Let $\mathcal{E}$  be an algebraic connection on $U$. Let us denote by $D_1, \dots, D_n$  the irreducible components of $D$ meeting $f^{-1}(0)$ and let $r_i(\mathcal{E})$ be the highest generic slope of $\mathcal{E}$ along $D_i$.\\ \indent
As an application of theorem \ref{theoremprincipalraffiné},
we prove the following
\begin{theorem}\label{GMirr}
If $\mathcal{E}$ has good formal structure along $D$ and if the fibers  $X_t$, $t\neq 0$ of $f$ are non characteristic at infinity\footnote{this is for example the case if $D$ is smooth and if the fibers of $f$ are transverse to $D$.} for  $\mathcal{E}$,  then the periods of the family $(\mathcal{E}_t)_{t\neq 0}$ are solutions of a system of linear polynomial differential equations whose slopes at $0$ are
$\leq r_1(\mathcal{E})+\cdots +r_n(\mathcal{E})$.
\end{theorem}
In the case where $\mathcal{E}$ is the trivial connection, we recover that the periods of a proper generically smooth family of algebraic varieties are solutions of a regular singular differential equation with polynomial coefficients \cite{KatzIHES}\cite{Del}. \\ \indent

The role played in this paper by nearby cycles has Verdier specialization \cite{SpeVerdier} and moderate nearby cycles \cite[XIII]{SGA7-2} as  $\ell$-adic counterparts. Let $V$ be an algebraic scheme over a perfect field $k$ of characteristic $p>0$, and let  $\overline{V}$ be  a compactification of $V$. Let $j: V\longrightarrow \overline{V}$ be the canonical inclusion, and $f\in \mathcal{O}_{\overline{V}}$. Let $S$ be the strict 
henselianization of 
$\mathcal{O}_{\mathds{A}^{1}_{k},0}$, denote by $\eta$ the generic point of $S$ and let us choose a geometric point $\overline{\eta}$ over $\eta$. Let $P$ be the wild ramification group of $\pi_1(\eta, \overline{\eta})$. Define $\overline{V}_S:=\overline{V}\times_{\mathds{A}^{1}_{k}} S$, $f_S:\overline{V}_S\longrightarrow S$ the base change of $f$ to $S$, $f_\eta: \overline{V}_\eta\longrightarrow \eta$ the restriction of $f_S$ over $\eta$ and $\iota:  \overline{V}_S\longrightarrow \overline{V}$ the canonical morphism.  Let $\mathcal{F}$ be a $\ell$-adic complex on $V$ with bounded constructible cohomology.
We say that $r\in \mathds{Q}_{\geq 0}$ is a nearby slope\footnote{Or a Verdier slope if Verdier specialization is used instead of moderate nearby cycles.} for $\mathcal{F}$  associated to $(\overline{V},f)$ if one can find a constructible $\ell$-adic sheaf $N$ on $\eta_{\et}$ with slope $r$ such that 
$$
(\psi_f(\iota^{\ast}j_{!}\mathcal{F}\otimes f_\eta^{\ast}N))^{P}\neq  0
$$
Hence, conjecture \ref{conj0} has a $\ell$-adic analogue that may be worth investigating. This is the following
\begin{question}
Is it true that the set of nearby slopes of  $\mathcal{F}$ is bounded?
\end{question}
That nearby slopes do not depend on the choice of a compactification is not clear to the author. Nor that a smoothness assumption on $V$ is needed. In any case, this leads to a notion of $\ell$-adic tame complex in the sense of Verdier specialization or moderate nearby cycles. As an analogue of the regularity of $\mathcal{O}_X$ in the theory of $\mathcal{D}$-modules, we raise the following
\begin{question}
Is it true that the constant sheaf $\overline{\mathds{Q}}_{\ell}$ on $V$ is tame in the sense of moderate nearby cycles? That is, that for every couple $(\overline{V},f)$ as above and for every constructible $\ell$-adic sheaf $N$ on $\eta_{\et}$ with slope $>0$, we have
$$
(\psi_f(\iota^{\ast}j_{!}\overline{\mathds{Q}}_{\ell}\otimes f_\eta^{\ast}N))^{P}\simeq 0   \text{ ?}
$$
\end{question}

Conjecture \ref{conj0} first appears in \cite{ProgCNRS}. This paper grew out an attempt to prove it. I thank Pierre Deligne, Zoghman Mebkhout and Claude Sabbah for valuable comments on this manuscript and Marco Hien for mentioning \cite{HR}, which inspired me a statement in the spirit of  theorem \ref{GMirr} and reignited my interest for a proof of theorem \ref{theoremprincipal}. 
This work has been achieved with the support of Freie Universität/Hebrew University of Jerusalem joint post-doctoral program and ERC 226257 program. I thank Hélène Esnault and Yakov Varshavsky for their support.

\section{Notations}
We collect here a few definitions  used all along this paper. The letter $X$ will denote a complex manifold.
\subsection{}
For a morphism $f:Y\longrightarrow X$ with $Y$ a  complex manifold, we denote by $f^{+}:D^{b}_{\hol}(\mathcal{D}_X)\longrightarrow D^{b}_{\hol}(\mathcal{D}_Y)$ and $f_{+}:D^{b}_{\hol}(\mathcal{D}_Y)\longrightarrow D^{b}_{\hol}(\mathcal{D}_X)$ the inverse image and direct image functors for  $\mathcal{D}$-modules. We note $f^{\dag}$ for $f^{+}[\dim Y-\dim X]$.
\subsection{}
Let $\mathcal{M}\in \mathcal{D}_{\hol}^b(X)$ and $f\in \mathcal{O}_X$. 
From
$
\mathcal{H}^{k}\psi_{f}(\mathcal{M}\otimes f^{+}N)\simeq 
\psi_{f}(\mathcal{H}^{k}\mathcal{M}\otimes f^{+}N)
$
for every  $k$, we deduce
\begin{equation}\label{slopecomplex}
\Sl_{f}^{\nb}(\mathcal{M})=\bigcup_{k}\Sl_{f}^{\nb}(\mathcal{H}^{k}\mathcal{M})
\end{equation}
Let us define 
$
\Sl^{\nb}(\mathcal{M}):=\bigcup_{f\in \mathcal{O}_X} \Sl_{f}^{\nb}(\mathcal{M})
$. The elements of $\Sl^{\nb}(\mathcal{M})$ are the \textit{nearby slopes}  of $\mathcal{M}$.
For  $S\subset \mathds{Q}_{\geq 0}$, we denote by $ \mathcal{D}_{\hol}^b(X)_{S}$ the full subcategory of  $\mathcal{D}_{\hol}^b(X)$ of complexes whose nearby slopes are in $S$.
\subsection{}\label{DRetSol}
Let us denote by $\DR:D^{b}_{\hol}(\mathcal{D}_X) \longrightarrow D^{b}_{c}(X,\mathds{C})$ the \textit{de Rham functor}\footnote{In this paper, we follow Hien's convention \cite{HienInv} according to which for a holonomic module $\mathcal{M}$, the complex $\DR \mathcal{M}$ is concentrated in degrees $0,\dots, \dim X$.} and by $\Sol:D^{b}_{\hol}(\mathcal{D}_X) \longrightarrow D^{b}_{c}(X,\mathds{C})$  the \textit{solution functor} for holonomic $\mathcal{D}_X$-modules. 
\subsection{}\label{localisation}
For every analytic subspace $Z$ in  $X$, we denote by $i_Z:Z\hookrightarrow X$ the canonical inclusion. The \textit{local cohomology triangle} for $Z$ and $\mathcal{M}\in \mathcal{D}_{\hol}^b(X)$ reads
\begin{equation}\label{ocholocal}
\xymatrix{
R\Gamma_{[Z]}\mathcal{M}\ar[r]&  \mathcal{M} \ar[r]& R\mathcal{M}(\ast Z)\ar[r]^-{+1}& 
}
\end{equation}
It is a distinguished triangle in  $D^{b}_{\hol}(\mathcal{D}_X)$.
The complex $R\Gamma_{[Z]}\mathcal{M}$ is  \textit{the local algebraic cohomology} of $\mathcal{M}$ along $Z$ and $R\mathcal{M}(\ast Z)$ is the \textit{localization} of $\mathcal{M}$ along $Z$. 

\subsection{}\label{bonnedecompositionformelle}
Let $\mathcal{M}$ be a germ of meromorphic connection at the origin of $\mathds{C}^{n}$. Let  $D$ be the pole  locus of  $\mathcal{M}$. For $x\in D$, we define $\hat{\mathcal{M}}_x:=\hat{\mathcal{O}}_{\mathds{C}^{n},x}\otimes_{\mathcal{O}_{\mathds{C}^{n},x}}\mathcal{M}$. We say that $\mathcal{M}$ has  \textit{good formal structure} if 
\begin{enumerate}
\item $D$ is a normal crossing divisor.
\item  For every $x\in D$, one can find coordinates $(x_1,\dots, x_n)$ centred at  $x$ with $D$ defined by $x_1\cdots x_i=0$, and an integer $p\geq 1$ such that if $\rho$ is the morphism $(x_1,\dots, x_n)\longrightarrow (x_1^p,\dots, x_i^p,x_{i+1}, \dots, x_n)$, we have a decomposition
\begin{equation}\label{decomposition}
\rho^{+}\hat{\mathcal{M}}_x\simeq \displaystyle{\bigoplus_{\varphi \in \mathcal{O}_{\mathds{C}^{n}}(\ast D)/\mathcal{O}_{\mathds{C}^{n}}}}
 \mathcal{E}^{\varphi}\otimes \mathcal{R}_{\varphi}
\end{equation}
where $\mathcal{E}^{\varphi}=(\hat{\mathcal{O}}_{\mathds{C}^{n},x}(\ast D),d+d\varphi)$ and $\mathcal{R}_{\varphi}$ is a meromorphic connection  with regular singularity along $D$.
\item\label{conditionphi}  For all $\varphi\in  \mathcal{O}_{\mathds{C}^{n}}(\ast D)/\mathcal{O}_{\mathds{C}^{n}}$ contributing to \eqref{decomposition}, we have $\divi \varphi \leq 0$.
\end{enumerate}
Let us remark that classically, one asks for condition \eqref{conditionphi} to be also true for the differences of two $\varphi$  intervening in \eqref{decomposition}. We won't impose this extra condition in this paper.
\subsection{} \label{divplusgrandepente}
Let  $\mathcal{M}$ be a meromorphic connection on  $X$ such that the pole locus $D$ of $\mathcal{M}$  has only a finite number of irreducible components $D_1,\dots, D_n$. For every $i=1,\dots, n$, we denote by $r_{D_i}(\mathcal{M})$ the highest generic slope of   $\mathcal{M}$ along $D_i$. We define the \textit{divisor of highest generic slopes of $\mathcal{M}$}  by
$$
r_{D_1}(\mathcal{M})D_1+\dots +r_{D_n}(\mathcal{M})D_n\in Z(X)_{\mathds{Q}}
$$

\section{Preliminaries on nearby cycles in the case of good formal structure}
\subsection{}
Let $n$ be an integer and take $i\in \mathds{N}^{\llbracket 1,n\rrbracket}$. The \textit{support of $i$} is the set of  $k\in \llbracket 1,n\rrbracket$ such that $i_k\neq 0$. If $E\subset \llbracket 1,n\rrbracket$, we define $i_E$ by $i_{Ek}=i_k$ for $k\in E$ and $i_{Ek}=0$ if  $k\notin E$.

\subsection{}\label{moduleelementaire}
Let $R$ be a  regular $\mathds{C}((t))$-differential module, and take $\varphi\in \mathds{C}[t^{-1}]$. For every $n\geq 1$, we define $\rho: t\longrightarrow t^{p}=x$ and
$$
\El(\rho, \varphi, R):=\rho_{+}(\mathcal{E}^{\varphi}\otimes R)
$$
If $R$ is the trivial rank 1 module, we will use the notation $\El(\rho, \varphi)$.
In general, $\El(\rho, \varphi, R)$ has slope $\ord \varphi/p$. The $\mathds{C}((x))$-modules of type $\El(\rho, \varphi, R)$ for variable $(\rho, \varphi, R)$ are called \textit{elementary modules}. From \cite[3.3]{SabbahpourLE}, we know that every $\mathds{C}((x))$-differential module can be written as a direct sum of elementary modules.
\subsection{Dimension 1}
In this paragraph, we work in a neighbourhood of the origin  $0\in \mathds{C}$. Let $x$ be a coordinate on $\mathds{C}$. Take $p\geq 1$ and define $\rho: x\longrightarrow t=x^{p}$.
\begin{proposition}\label{propdim1}
Let $\mathcal{M}$ be a germ of holonomic $\mathcal{D}$-module at the origin. Let $r>0$ be a rational number. The following conditions are equivalent
\begin{enumerate}
\item The rational $r$ is not a slope for  $\mathcal{M}$ at $0$.
\item For every germ $N$ of meromorphic connection of slope $r/p$, we have
$$
\psi_{\rho}(\mathcal{M}\otimes \rho^{+}N)\simeq 0
$$
\end{enumerate}
\end{proposition}
\begin{proof}
Since $\psi$  is not sensitive to localization and formalization, one can work formally at 0 and suppose that  $\mathcal{M}$ and $N$ are differential $\mathds{C}((x))$-modules. \\ \indent
Let us prove $(2)\Longrightarrow (1)$ by contraposition. Define $\rho^{\prime}: u\longrightarrow u^{p^{\prime}}=x$, $\varphi(u)\in \mathds{C}[u^{-1}]$ with $q=\ord\varphi(u)$ and $R$ a $\mathds{C}((u))$-regular module such that $\El(\rho^{\prime}, \varphi(u), R)$ is a non zero elementary factor  \ref{moduleelementaire} of $\mathcal{M}$ with slope $r=q/p$. Define
$$
N:=\rho_+\El(\rho^{\prime}, -\varphi(u))=\El(\rho\rho^{\prime}, -\varphi(u))
$$
The module  $N$ has slope $q/pp^{\prime}=r/p$. A direct factor of $\psi_{\rho}(\mathcal{M}\otimes \rho^{+}N)$ is
\begin{align*}
\psi_{\rho}(\rho_{+}^{\prime}(\mathcal{E}^{\varphi}\otimes R)\otimes \rho^{+} N))&\simeq
\psi_{\rho}(\rho_{+}^{\prime}(\mathcal{E}^{\varphi}\otimes R)\otimes \rho^{+}\El(\rho\rho^{\prime}, -\varphi(u)))
 \\
 & \simeq \psi_{\rho}(\rho_{+}^{\prime}(\mathcal{E}^{\varphi}\otimes R\otimes (\rho\rho^{\prime})^{+}\El(\rho\rho^{\prime}, -\varphi(u)))\\
 &\simeq
 \psi_{\rho\rho^{\prime}}(\mathcal{E}^{\varphi}\otimes R\otimes (\rho\rho^{\prime})^{+}\El(\rho\rho^{\prime}, -\varphi(u)))
\end{align*}
where the last identification comes from the compatibility of $\psi$ with proper direct image. By \cite[2.4]{SabbahpourLE}, we have
$$
(\rho\rho^{\prime})^{+}\El(\rho\rho^{\prime}, -\varphi(u))\simeq \displaystyle{\bigoplus_{\zeta^{pp^{\prime}}=1}}\mathcal{E}^{-\varphi(\zeta u)}
$$
So $\psi_{\rho\rho^{\prime}} R$ is a direct factor of $\psi_{\rho}(\mathcal{M}\otimes \rho^{+}N)$ of rank $np(\rg R)>0$, and $(2)\Longrightarrow (1)$ is proved. \\ \indent
Let us prove $(1)\Longrightarrow (2)$. Let $N$ be a $\mathds{C}((t))$-differential module of slope $r/p$. Then $\rho^{+}N$ has slope $r$. Thus, the slopes of $\mathcal{M}\otimes \rho^{+}N$ are $> 0$. Hence, it is enough to show the following
\begin{lemme}
Let $M$ be a $\mathds{C}((x))$-differential module whose slopes are $>0$. Then, we have $\psi_{\rho}M\simeq 0$.
\end{lemme}
By Levelt-Turrittin decomposition, we are left to study the case where $M$ is a direct sum of modules of type $\mathcal{E}^{\varphi}\otimes R$, where $\varphi\in \mathds{C}[x^{-1}]$ and where $R$ is a regular $\mathds{C}((x))$-module. The hypothesis on the slopes of $M$ implies $\varphi\neq 0$, and the expected vanishing is standard. 
\end{proof}

\subsection{A vanishing criterion}\label{cycleprochbonne}
Let $\mathcal{M}$  be a germ of meromorphic connection at the origin $0\in \mathds{C}^{n}$. We suppose that $\mathcal{M}$ has  good formal structure at $0$. Let $D$ be the pole locus of $\mathcal{M}$. Let $\rho_p$ be a ramification of degree $p$ along the components of  $D$ as in  \eqref{decomposition}.
\begin{proposition}\label{lemmeannulation}
Let $f\in \mathcal{O}_{\mathds{C}^{n},0}$. Let us define $Z:=\divi f$ and suppose that $|Z|\subset D$.  We suppose that for every irreducible component $E$ of $|Z|$,  we have  $$r_E(\mathcal{M})\leq  r v_E(f)$$
Then  for every germ $N$ of meromorphic connection  at 0 with slopes $>r$, we have
\begin{equation}\label{psinulle}
\psi_{f}(\mathcal{M}\otimes f^{+}N)\simeq 0
\end{equation}
in a neighbourhood of $0$.
\end{proposition}
\begin{proof}
Let us choose local coordinates $(x_1,\dots, x_n)$ and $a \in \mathds{N}^{n}$ such that $f$ is the function $x\longrightarrow x^a$.
Take $N$ with slopes $> r$. One can always suppose that $N$ is a $\mathds{C}((t))$-differential module and $p=qk$ where $\rho^{\prime}:t\longrightarrow t^{k}$ decomposes  $N$.\\ \indent
The morphism $\rho_p$ is a finite cover away from $D$, so the canonical adjunction morphism
\begin{equation}\label{morphsurjectif}
\xymatrix{
\rho_{p+}\rho^{+}_p \mathcal{M} \ar[r]& \mathcal{M} 
}
\end{equation}
is surjective away from $D$. So the cokernel of \eqref{morphsurjectif} has support in  $D$. From \cite[3.6-4]{Mehbsmf}, we know that both sides of  \eqref{morphsurjectif} are localized along $D$. So \eqref{morphsurjectif} is surjective. We thus have to prove
\begin{equation}\label{annulationvoulue}
\psi_{f\rho_p}(\rho^{ +}_p\mathcal{M}\otimes (f\rho_p)^{+}N)\simeq 0
\end{equation}
Since $|Z|\subset D$, we have $f\rho_p=\rho^{\prime}f\rho_q$. So the left hand side of \eqref{annulationvoulue} is a direct sum of  $k$ copies of 
\begin{equation}\label{kcopies}
\psi_{f\rho_q}(\rho^{ +}_p\mathcal{M}\otimes (f\rho_p)^{+}N)
\end{equation}
We thus have to prove that \eqref{kcopies} is 0 in a neighbourhood of 0. We have
$$
(f\rho_p)^{+}N\simeq (f\rho_q )^{+}\rho^{\prime +}N
$$
with $\rho^{\prime +}N$ decomposed with slopes $>rk$. The zero locus of  $f\rho_q$ is $|Z|$, and if $E$ is an irreducible component of $|Z|$, the highest generic slope of  $\rho^{+}_p\mathcal{M}$ along $E$ is $$r_{E}(\rho^{+}_p\mathcal{M})=p \cdot r_{E}(\mathcal{M})\leq  r k  \cdot q \cdot v_{E} (f)=rk \cdot v_{E} (f\rho_q)$$
Hence we can suppose that $\rho_p=\id$ and that $N$ is decomposed. \\ \indent
Take
$$
N=\mathcal{E}^{P(t)/t^{m}}\otimes R
$$
with $P(t)\in \mathds{C}[t]$ satisfying $P(0)\neq 0$, with $m>r$ and with $R$ regular. Since $\psi$ is insensitive to formalization, one can suppose
$$
\mathcal{M}=\mathcal{E}^{\varphi(x)}\otimes \mathcal{R}
$$
with  $\varphi(x)$ as in \ref{bonnedecompositionformelle} $(3)$ and $\mathcal{R}$ regular. By Sabbah-Mochizuki theorem, the multiplicity of $-\divi\varphi(x)$ along a component $D^{\prime}$ of $D$ is a generic slope of  $\mathcal{M}$ along $D^{\prime}$. Thus, one can write $\varphi(x)=g(x)/x^b$ where $g(0)\neq 0$ and  where the $b_i$ are such that if $i\in \Supp a$, we have $b_i\leq r a_i<ma_i$. 
We thus have to prove the
\begin{lemme}\label{lemmeannulation2}
Take $g, h\in \mathcal{O}_{\mathds{C}^{n},0}$ such that $g(0)\neq 0$ and $h(0)\neq 0$. Let $\mathcal{R}$ be a regular meromorphic connection with poles contained in $x_1\cdots x_n=0$. Take $a,b\in \mathds{N}^{\llbracket 1,n\rrbracket}$ such that $A:=\Supp a$ is non empty and $b_i< a_i$ for every $i\in A$. Then
$$
\psi_{x^a}(\mathcal{E}^{g(x)/x^b+h(x) /x^{a}}\otimes \mathcal{R})\simeq 0
$$
in a neighbourhood of 0. 
\end{lemme}
\end{proof}
\subsection{Proof of \ref{lemmeannulation2}}
We define $\mathcal{M} :=\mathcal{E}^{g(x)/x^b+h(x) /x^{a}}\otimes \mathcal{R}$. Since $A$ is not empty, a change of variable allows one to suppose $h=1$. If $\Supp b\subset A$, a change of variable shows that  \ref{lemmeannulation2} is a consequence of \ref{dernierlemmeannulation}. Let $i\in \Supp b$ be an integer such that $i\notin A$. Using $x_i$, a change of variable allows one to suppose  $g=1$.
Let $p_1,\dots, p_n\in \mathds{N}^{\ast}$ such that $a_j p_j$ is independent from $j$ for every $j\in A$  and $p_j=1$ if $j\notin A$. Let $\rho_p$ be the morphism $x\longrightarrow x^p$. Like in  \eqref{morphsurjectif}, we see that 
$$
\xymatrix{
\rho_{p+}\rho^{+}_p \mathcal{M} \ar[r]& \mathcal{M} 
}
$$
is surjective. We are thus left to prove that \ref{lemmeannulation2} holds for multi-indices $a$ such that $a_j$ does not depend on  $j$ for every $j\in A$. Let us denote by $\mathds{1}_A$ the characteristic function of $A$. From \cite[3.3.13]{PTM}, it is enough to prove
$$
\psi_{x^{\mathds{1}_A}}(\mathcal{E}^{1/x^b+1 /x^{a}}\otimes \mathcal{R})\simeq 0
$$
Using the fact that $\mathcal{R}$ is a successive extension of regular modules of rank 1, one can suppose that $\mathcal{R}=x^{c}$, where $c\in \mathds{C}^{\llbracket 1,n\rrbracket}$. Let 
\[
\xymatrix{ 
\mathds{C}^{n} \ar@{^{(}->}[r]^-{i} \ar[rd]_-{x^{\mathds{1}_A}}& \mathds{C}^{n} \times \mathds{C}\ar[d]   \\
                                                                        &       \mathds{C}
}
\] 
be the inclusion given by the graph of $x\longrightarrow x^{\mathds{1}_A}$. Let $t$ be a coordinate on the second factor of $\mathds{C}^{n} \times \mathds{C}$. We have to prove
$$
\psi_{t}(i_+(x^{c}\mathcal{E}^{1/x^b+1 /x^{a}}))\simeq 0
$$
Define $\delta:=\delta(t-x^{\mathds{1}_A})\in i_+(x^{c}\mathcal{E}^{1/x^b+1 /x^{a}})$ and let $(V_k)_{k\in \mathds{Z}}$ be the Kashiwara-Malgrange filtration on  $\mathcal{D}_{\mathds{C}^{n} \times \mathds{C}}$ relative to $t$. 
For $d\in \mathds{N}^{\llbracket 1,n\rrbracket}$  such that $x^d=0$ is the pole locus of $x^{c}\mathcal{E}^{1/x^b+1 /x^{a}}$, the family of sections $x^d$ generates $x^{c}\mathcal{E}^{1/x^b+1 /x^{a}}$. For such $d$, the family  $s:=x^d\delta$ generates $i_+(x^{c}\mathcal{E}^{1/x^b+1 /x^{a}})$. We are left to prove $s\in V_{-1}s$. One can always suppose that $1\in A$.
$$
x_1\partial_1 s=(d_1+c_1)s-\frac{b_1}{x^{b}}s-\frac{a_1}{x^{a}}s- x^{\mathds{1}_A}\partial_t s
$$
We define $M\in \mathds{N}^{\llbracket 1,n\rrbracket}$ by $M_k=\max(a_k, b_k)$ for every $k\in  \llbracket 1,n\rrbracket$. We thus have
\begin{equation}\label{eq0}
x^{M}x_1\partial_1 s=(d_1+c_1)x^{M}s-b_1x^{M-b}s-a_1x^{M-a}s- x^{M} x^{\mathds{1}_A}\partial_t s
\end{equation}
We have $M=a+b_{A^{c}}=\mathds{1}_A+ (a-\mathds{1}_A)+b_{A^{c}}=\mathds{1}_A+b+m$ with $m\in \mathds{N}^{\llbracket 1,n\rrbracket}$. So $$x^{M-b}s=x^m ts\in V_{-1}s$$ 
Moreover, we have
$$
x^{M}x_1\partial_1 s=x_1\partial_1 x^M s-M_1 x^M s=x_1\partial_1 x^{m+b} ts-M_1 x^{m+b} ts\in V_{-1}s
$$
and
$$
x^{M} x^{\mathds{1}_A}\partial_t s=x^{m+b}\partial_t  x^{2\times\mathds{1}_A}s=x^{m+b}\partial_t t^{2}s =2x^{m+b} ts+x^{m+b}t(t\partial_t) s \in V_{-1}s
$$
So \eqref{eq0} gives
\begin{equation}\label{eq1}
x^{M-a}s \in V_{-1}s
\end{equation}
Let us recall that $i$ is such that $i\notin A$ and $i\in \Supp b$. In particular  $(M-a)_i=b_i\neq 0$ and $\partial_i \delta=0$. Applying $x_i\partial_i$ to \eqref{eq1}, we obtain
$$
(d_i+c_i+b_i) x^{M-a}s -b_i\frac{ x^{M-a}}{x^b}s\in V_{-1}s
$$
so from \eqref{eq1}, we deduce $x^{M-a-b}s\in V_{-1}s$ . We have $M-a-b=-b_A$, so by multiplying $x^{M-a-b}s$ by $x^{b_A}$, we get $s\in V_{-1}s$.

\subsection{}
The aim of this paragraph is to prove the following
\begin{lemme}\label{dernierlemmeannulation}
Let $\alpha,a\in \mathds{N}^{\llbracket 1,n\rrbracket}$ such that $\Supp \alpha$ is not empty and $\Supp \alpha\subset \Supp a$. Let $\mathcal{R}$ be a regular meromorphic connection with poles contained in  $x_1\cdots x_n=0$.  We have
$$
\psi_{x^\alpha}(\mathcal{E}^{1/x^{a}}\otimes \mathcal{R})\simeq 0
$$
\end{lemme}
\begin{proof}
Let $p_1,\dots, p_n$ be integers such that $\alpha_i p_i$ does not depend of  $i$ for every $i\in \Supp \alpha$ (we denote by $m$ this integer) and $p_i=1$ if $i\neq \Supp\alpha$. Let $\rho_p$ be the morphism $x\longrightarrow x^p$. Like in \eqref{morphsurjectif}, the morphism
$
\rho_{p+}\rho^{+}_p \mathcal{M} \longrightarrow  \mathcal{M} 
$
is surjective. We are left to prove \ref{dernierlemmeannulation} for $\alpha$ such that $\alpha_i$ does not depend of $i$ for every  $i\in \Supp \alpha$. From \cite[3.3.13]{PTM}, one can suppose $\alpha_i=1$ for every  $i\in \Supp \alpha$. So $\alpha\leq a$.\\ \indent
One can suppose  $\mathcal{R}=x^{b}$ where $b\in \mathds{N}^{\llbracket 1,n\rrbracket}$.
Let
\[
\xymatrix{ 
\mathds{C}^{n} \ar@{^{(}->}[r]^-{i} \ar[rd]_-{x^{\alpha}}& \mathds{C}^{n} \times \mathds{C}\ar[d]   \\
                                                                        &       \mathds{C}
}
\] 
be the inclusion given by the graph of $x\longrightarrow x^{\alpha}$. Let $t$ be a coordinate on the second factor of $\mathds{C}^{n} \times \mathds{C}$. We have to show
$$
\psi_{t}(i_+(x^{b}\mathcal{E}^{1/x^{a}}))\simeq 0
$$
Define $\delta:=\delta(t-x^{\alpha})\in i_+(x^{b}\mathcal{E}^{1/x^{a}})$.
For $c\in \mathds{N}^{\llbracket 1,n\rrbracket}$ such that $\Supp c\subset \Supp a\cup \Supp b$, the family of sections $x^c$ generates $x^{b}\mathcal{E}^{1/x^{a}}$. For such $c$, the family $s:=x^c\delta$ generates $i_+(x^{b}\mathcal{E}^{1/x^{a}})$. It is thus enough to show $s\in V_{-1}s$. Let us choose $i\in \Supp \alpha$. We have
$$
x_i\partial_i s=(c_i+b_i)s-\frac{a_i}{x^{a}}s- x^{\alpha}\partial_t s
$$
We have $\alpha\leq a$. Define $a=\alpha+a^{\prime}$. From
$$
x^{\alpha}x_i\partial_i s =x_i\partial_i x^{\alpha} s-x^{\alpha} s=x_i\partial_i t s-ts\in V_{-1}s
$$
we deduce that $a_i s+x^{a^{\prime}}x^{2\alpha}\partial_t s \in V_{-1}s$.
We also have $x^{2\alpha}\partial_t s=\partial_t x^{2\alpha} s=\partial_t t^{2} s=2ts+t (t\partial_t)s\in V_{-1}s$. Since $a_i\neq 0$, we deduce $s\in V_{-1}s$ and  \ref{dernierlemmeannulation} is proved.

\end{proof}
%

\section{Proof of theorem \ref{theoremprincipal} }
\subsection{Dévissage to the case of  meromorphic connections }\label{reduction}
Suppose that theorem \ref{theoremprincipal} is true for meromorphic connections for every choice of ambient manifold. Let us show that theorem \ref{theoremprincipal} is true for $\mathcal{M}\in \mathcal{D}_{\hol}^{b}(X)$. We argue by induction on  $\dim X$. The case where $X$ is a point is trivial. Let us suppose that $\dim X>0$. We define  $Y:=\Supp \mathcal{M}$ and we argue by induction on  $\dim Y$.
 \\ \indent
Let us suppose that $Y$ is a strict closed subset of $X$. We denote by $i:Y\longrightarrow X$ the canonical inclusion.  Let $\pi:\tilde{Y}\longrightarrow Y$ be a resolution of the singularities of $Y$ \cite{AHV} and $p:=i\pi$. The regular locus $\Reg Y$ of $Y$ is a dense open subset in  $Y$ and $\pi$ is an isomorphism above $\Reg Y$. By Kashiwara theorem, we deduce that the cone  $\mathcal{C}$ of the adjunction morphism
$$
\xymatrix{
p_{+}p^{\dag}\mathcal{M} \ar[r]& \mathcal{M}
}
$$
has support in $\Sing Y$, with $\Sing Y$ a strict closed subset in $Y$. Let $x\in X$ and let $B$ be a neighbourhood of $x$ with compact closure $\overline{B}$. Then, $p^{-1}(\overline{B})$ is compact. Since $\dim \tilde{Y}<\dim X$, theorem \ref{theoremprincipal} is true for $p^{\dag}\mathcal{M}\in \mathcal{D}_{\hol}^{b}(\tilde{Y})$. Let  $(U_i)$ be a finite family of open sets in $\tilde{Y}$ covering $p^{-1}(\overline{B})$ and such that for every $i$, the set $\Sl^{\nb}((p^{\dag}\mathcal{M})_{|U_i})$ is bounded by a rationnal $r_i$. Define $R=\max_i r_i$. \\ \indent
By induction hypothesis applied to $\mathcal{C}$, one can suppose at the cost of taking a smaller $B$ containing $x$
that the set  $\Sl^{\nb}(\mathcal{C}_{|B})$ is bounded by a rational $R^{\prime}$. Take $f\in \mathcal{O}_B$. We have a distinguished triangle
\begin{equation}\label{trianglede psi}
\xymatrix{
 \psi_{f}(p_{+}p^{\dag}\mathcal{M} \otimes f^{+}N) \ar[r]& \psi_{f}( \mathcal{M}\otimes f^{+}N)\ar[r]&  \psi_{f}(\mathcal{C}\otimes f^{+}N) \ar[r]^-{+1}&
}
\end{equation}
By projection formula and compatibility of $\psi$ with proper direct image, \eqref{trianglede psi} is isomorphic to 
$$
\xymatrix{
p_{+}\psi_{fp}(p^{\dag}\mathcal{M} \otimes (pf)^{+}N) \ar[r]& \psi_{f}( \mathcal{M}\otimes f^{+}N)\ar[r]&  \psi_{f}(\mathcal{C}\otimes f^{+}N) \ar[r]^-{+1}&
}
$$
So we have the desired vanishing on  $B$ for $r>\max(R,R^{\prime})$.
 \\ \indent
We are left with the case where $\dim \Supp \mathcal{M}=\dim X$. Let $Z$ be a hypersurface containing $\Sing \mathcal{M}$. We have a triangle
$$
\xymatrix{
R\Gamma_{[Z]}\mathcal{M} \ar[r]& \mathcal{M} \ar[r]& \mathcal{M}(\ast Z) \ar[r]^-{+1}&
}
$$
By applying the induction hypothesis to $R\Gamma_{[Z]}\mathcal{M}$, we are left to prove theorem  \ref{theoremprincipal} for $ \mathcal{M}(\ast Z) $. The module  $ \mathcal{M}(\ast Z) $ is a meromorphic connection, which concludes the reduction step.

\subsection{The case of  meromorphic connections}
At the cost of taking an open cover of $X$, let us take a resolution of turning points $p:\tilde{X}\longrightarrow X$ for $\mathcal{M}$ as given by Kedlaya-Mochizuki theorem. Let $D$ be the pole locus of $\mathcal{M}$. Since $p$ is an isomorphism above $X\setminus D$, the cone of
\begin{equation}\label{conemorphs0}
\xymatrix{
p_+ p^{+ }\mathcal{M} \ar[r]& \mathcal{M}
}
\end{equation}
has support in the pole locus $D$ of $\mathcal{M}$. 
From \cite[3.6-4]{Mehbsmf}, the left hand side of \eqref{conemorphs0} is localized along $D$. So \eqref{conemorphs0} is an isomorphism. We thus have a canonical isomorphism
$$
p_{+ }\psi_{fp}(p^{+ }\mathcal{M}\otimes (fp)^{+}N)\simeq
\psi_{f}(\mathcal{M}\otimes f^{+}N) 
$$
Since $p$ is proper, we see as in \ref{reduction} that we are left to prove theorem \ref{theoremprincipal} for $p^{+ }\mathcal{M}$. We thus suppose that $\mathcal{M}$ has a good formal structure. At the cost of taking an open cover, we can suppose that $D$ has only a finite number of irreducible components. Let $S$ be the divisor of highest generic slopes \ref{divplusgrandepente} of $\mathcal{M}$. Let $S_1, \dots, S_m$ be the irreducible components of $|S|$. Let us prove that $\Sl^{\nb}(\mathcal{M})$ is bounded by $\deg S$. It is is a local statement. Let $f\in \mathcal{O}_X$ and define $Z:=\divi f$. Let us denote by $|Z|$ (resp. $|S|$) the support of $Z$ (resp. $S$) and let us admit for a moment the validity of the following
\begin{proposition}\label{propositionprincipale}
Locally on $X$, one can find a proper birationnal morphism $\pi :\tilde{X}\longrightarrow X$ such that
\begin{enumerate}
\item\label{cond1prop} $\pi$ is an isomorphism above $X\setminus |Z|$.
\item\label{cond2prop}   $\pi^{-1}(|Z|)\cup \pi^{-1}(|S|)$  is a normal crossing divisor.
\item\label{cond3prop}  for every valuation $v_E$ measuring the vanishing order along an irreducible component $E$ of  $\pi^{-1}(|Z|)$, we have   
$$
v_E(S)\leq (\deg S) v_E(f)
$$
\end{enumerate}
\end{proposition}
Let us suppose that \ref{propositionprincipale} is true. At the cost of taking an open cover, let us take a morphism $\pi:\tilde{X}\longrightarrow X$  as in  \ref{propositionprincipale}. Since condition \eqref{cond1prop} is true, the cone of the canonical comparison morphism 
\begin{equation}\label{conemorphs}
\xymatrix{
\pi_+ \pi^{+ }\mathcal{M} \ar[r]& \mathcal{M}
}
\end{equation}
has support in $|Z|$. Since $f^{+}N$ is localized along  $|Z|$, we deduce that \eqref{conemorphs} induces an isomorphism
$$
\xymatrix{
(\pi_+ \pi^{+ }\mathcal{M}) \otimes f^{+}N   \ar[r]^-{\sim}& \mathcal{M}\otimes f^{+}N 
}
$$
Applying $\psi_f$ and using the fact that $\pi$ is proper, we see that it is enough to prove
\begin{equation}\label{ptitequation}
\psi_{f\pi}(\pi^{+}\mathcal{M}\otimes (f\pi)^{+}N)\simeq 0
\end{equation}
for every germ $N$ of meromorphic connection at the origin with slope $r >\deg S$. Since $(f\pi)^{+}N$ is localized along $\pi^{-1}(|Z|)$, the left-hand side of \eqref{ptitequation} is
\begin{equation}\label{psiapresresolution}
\psi_{f\pi}((\pi^{+}\mathcal{M})(\ast \pi^{-1}(|Z|))\otimes (f\pi)^{+}N)
\end{equation}
The vanishing of \eqref{psiapresresolution} is a local statement on $\tilde{X}$. Since \eqref{cond2prop}   and \eqref{cond3prop} are true, \ref{lemmeannulation} asserts that it is enough to show that for every irreducible component $E$ of $\pi^{-1}(|Z|)$, we have
$$r_E((\pi^{+}\mathcal{M})(\ast \pi^{-1}(|Z|)))\leq  (\deg S) v_E(f\pi)$$
Let us notice that $v_E(f\pi)=v_E(f)$. Let $P$ be a point in the smooth locus of $E$. Let 
$\varphi$ as in  \eqref{decomposition} for  $\mathcal{M}$ at the point $Q:=\pi(P)$. For $i=1,\dots, n$, let $t_i=0$ be an equation of $S_i$ in a neighbourhood of $Q$. Modulo a unit in  $\mathcal{O}_{X,Q}$, we have $\varphi=1/t_1^{r_1}\cdots t_n^{r_n}$ where $r_i\in \mathds{Q}_{\geq 0}$. If $u=0$ is a local equation  for $E$ in a neighbourhood of $P$, we have modulo a unit in  $\mathcal{O}_{\tilde{X},P}$
$$\varphi \pi=\frac{1}{u^{r_1v_E(t_1)}\cdots u^{r_n v_E(t_n)}}$$
So the slope of $\mathcal{E}^{\varphi \pi}(\ast \pi^{-1}(|Z|))$ along $E$ is $r_1v_E(t_1)+\cdots +r_n v_E(t_n)$. By  Sabbah-Mochizuki theorem, $r_i$ is a slope of $\mathcal{M}$ generically along $S_i$, so $r_i\leq r_{S_i}( \mathcal{M})$. We deduce that
$$
r_E(\pi^{+}\mathcal{M}(\ast \pi^{-1}(|Z|)))\leq 
\sum_i r_{S_i}( \mathcal{M}) v_E(t_i)=v_E(S)\leq  (\deg S) v_E(f)
$$
This concludes the proof of theorem  \ref{theoremprincipal} and theorem \ref{theoremprincipalraffiné}.

\subsection{Proof of \ref{propositionprincipale}}
At the cost of taking an open cover of $X$, let us take a finite sequence of blow-up
\begin{equation}\label{resolution}
\xymatrix{
\pi_n:X_n\ar[r]^-{p_{n-1}} & X_{n-1}\ar[r]^-{p_{n-2}} & \cdots  \ar[r]& X_1\ar[r]^-{p_0} & X_0=X
}
\end{equation}
given by 3.15 and 3.17 of \cite{BMUniformization} for $Z$ relatively to the normal crossing divisor $|S|$. Let $|Z|_i$ be the strict transform of $|Z|$ in $X_i$ and let $C_i$ be the center of $p_i$. We define inductively $H_0=H$ and $H_{i+1}=p_{i}^{-1}(H_{i})\cup p^{-1}_i(C_i)$ for $i=1,\dots, n$, where $p_{i}^{-1}$ denotes the set theoretic inverse image. In particular $H_{i+1}$ is a closed subset of $X_{i+1}$. We will endow it with its canonical reduced structure. Then, \eqref{resolution} satisfies \\ \indent
$(i)$ $C_i$ is a smooth closed subset of $|Z|_i$.\\ \indent
$(ii)$ $C_i$ is nowhere dense in $|Z|_i$.\\ \indent
$(iii)$ $C_i$ and $H_i$ have normal crossing for every  $i$.\\ \indent
$(iv)$ $|Z|_n\cup H_n$ is a normal crossing divisor.\\ \noindent
Since $C_i$ and the components of $H_i$ are reduced and smooth, condition $(iii)$ means that locally on $X_i$, one can find coordinates $(x_1,\dots, x_k)$ such that $H_i$ is given by the equation $x_1\cdots x_l=0$ and the ideal of $C_i$ is generated by some $x_j$ for $j=1,\dots, k$.
Using condition $(i)$, we see by induction that $\pi^{-1}_n(|Z|)\cup \pi^{-1}_n(|S|)=|Z|_n\cup H_n$. Proposition  \ref{propositionprincipale} is thus a consequence of
\begin{proposition}\label{casparticulier}
Let
$$
\xymatrix{
\pi_n:X_n\ar[r]^-{p_{n-1}} & X_{n-1}\ar[r]^-{p_{n-2}} & \cdots  \ar[r]& X_1\ar[r]^-{p_0} & X_0=X
}
$$
be a sequence of blow-up satisfying $(i)$,$(ii)$ and $(iii)$. For every irreducible component $E$ of $\pi^{-1}_n(|Z|)$, we have
\begin{equation}\label{inegalitevoulue}
v_E(S)\leq (\deg S) v_E(f)
\end{equation}
\end{proposition}
\begin{proof}
Let $S_1, \dots, S_m$ be the irreducible components of $|S|$ and let $Z_1, \dots, Z_{m^{\prime}}$ be the irreducible components of $Z$. Note that some $Z_i$ can be in  $|S|$. We define $a_i=v_{Z_i}(f)>0$ and let $Z_{ji}$ (resp. $S_{ji}$) be the strict transform of $Z_j$ (resp. $S_j$) in $X_i$.\\ \indent
We argue by induction on  $n$. If $n=0$, $E$ is one of the $Z_i$ and then \eqref{inegalitevoulue} is obvious. We suppose that \eqref{inegalitevoulue} is true for a composite of $n$ blow-up and we prove that \eqref{inegalitevoulue}  is true for a composite of $n+1$ blow-up. \\ \indent
Let  $\mathcal{C}_n$ be the set of irreducible components of $$\displaystyle{\bigcup_{i=0}^{n-1} (p_{n-1}\cdots p_i)^{-1}(C_i)}$$
Each element $E\in \mathcal{C}_n$ will be endowed with its reduced structure. Condition $(i)$ implies that the irreducible components of  $\pi_n^{\ast}Z$ are the $Z_{in}$ and the elements of $\mathcal{C}_n$. Condition $(ii)$ implies that none of the $Z_{in}$ belongs to $\mathcal{C}_n$. Thus, we have 
$$
\pi_n^{\ast}Z=\divi f\pi_n=a_1 Z_{1n}+\cdots +a_{m^{\prime}}Z_{m^{\prime}n}+\displaystyle{\sum_{E\in \mathcal{C}_n}} v_E(f) E
$$
On the other hand, we have
$$
\pi_n^{\ast}S=r_{S_1}( \mathcal{M})  S_{1n}+\cdots+ r_{S_m}( \mathcal{M})  S_{mn}+\displaystyle{\sum_{E\in \mathcal{C}_n}} v_{E}(S) E
$$
Let us consider the last blow-up
$p_{n}: X_{n+1}\longrightarrow X_{n}$. Let us denote by $P$ the exceptionnal divisor of $p_{n}$ and let $E_{n+1}$ be the strict transform of $E\in \mathcal{C}_n$ in $X_{n+1}$. 
We have
$$
p_{n}^{\ast} Z_{in}=Z_{in+1}+\alpha_i P \text{\quad \quad with $\alpha_i\in \mathds{N}$}
$$
Since 
$$
H_n=\displaystyle{\bigcup_{j=0}^{m}} S_{jn}  \cup \displaystyle{\bigcup_{E\in \mathcal{C}_n}} E
$$
we deduce from condition $(iii)$ and smoothness of $C_n$ that
$$
p_{n}^{\ast} E=E_{n+1}+ \epsilon_E P
\text{\quad \quad with $\epsilon_E\in \{0,1\}$}
$$
and 
$$
p_{n}^{\ast} S_{in}=S_{in+1}+ \epsilon_i P  \text{\quad \quad with $\epsilon_i\in \{0,1\}$}
$$
Hence, we have
$$
\pi_{n}^{\ast}Z=\sum  a_iZ_{in+1}+\displaystyle{\sum_{E\in \mathcal{C}_n}} v_E(f) E_{n+1}+(\sum a_i\alpha_i+\sum_{E\in \mathcal{C}_n} \epsilon_E v_E(f))P
$$
and
$$
\pi_{n}^{\ast}S=\sum r_{S_i}( \mathcal{M})  S_{in+1}+\displaystyle{\sum_{E\in \mathcal{C}_n}} v_E(S) E_{n+1}+ (\sum r_{S_i}( \mathcal{M}) \epsilon_i+\sum_{E\in \mathcal{C}_n} \epsilon_E v_{E}(S))P
$$
Formula \eqref{inegalitevoulue} is true for the $Z_{in+1}$. By induction hypothesis, formula \eqref{inegalitevoulue} is true for $E_{n+1}$, where $E\in \mathcal{C}_n$. We are left to prove that \eqref{inegalitevoulue}  is true for $P$. Conditions $(i)$ and $(ii)$ imply that one of the $\alpha_i$ is non zero, so
\begin{align*}
(\deg S)\left(\sum a_i\alpha_i+\sum \epsilon_E v_E(f)\right) &\geq (\deg S)+(\deg S)\sum \epsilon_E v_E(f)  \\ 
&\geq \sum r_{S_i}( \mathcal{M})  \epsilon_i+\sum \epsilon_E (\deg S)v_E(f) \\ 
&\geq \sum r_{S_i}( \mathcal{M}) \epsilon_i+\sum \epsilon_E v_E(S)
\end{align*}
\end{proof}

\section{Duality}
We prove theorem \ref{theoremprincipal-1} $(i)$. Let us denote by $\mathds{D}$ the duality functor for $\mathcal{D}$-modules. There is a canonical comparison morphism
\begin{equation}\label{comparaison}
\xymatrix{
\mathds{D}(\mathcal{M}\otimes f^{+}N)\ar[r]& \mathds{D}\mathcal{M}  \otimes f^{+}\mathds{D} N
}
\end{equation}
On a punctured neighbourhood of $0\in \mathds{C}$, the module $N$ is isomorphic to a finite sum  of copies of the trivial connection. Thus, there is a neighbourhood $U$  of $Z$ such that the restriction of \eqref{comparaison}  to $U\setminus Z$ is an isomorphism. Hence, the cone of \eqref{comparaison} has support in $Z$. We deduce that
$$
\xymatrix{
(\mathds{D}(\mathcal{M}\otimes f^{+}N))(\ast Z)\ar[r]^-{\sim}&\mathds{D} \mathcal{M} \otimes f^{+}((\mathds{D}N)(\ast 0))
}
$$
We have $(\mathds{D}N)(\ast 0)\simeq N^{\ast}$, where $\ast$ is the duality functor for meromorphic connection. Note that $\ast$ is a slope preserving involution. Since  nearby cycles are insensitive to localization and commute with duality for $\mathcal{D}$-modules, we have 
$$
\psi_f(\mathds{D}\mathcal{M}  \otimes f^{+}N^{\ast})\simeq \mathds{D}(\psi_f(\mathcal{M}\otimes f^{+}N))
$$
and theorem \ref{theoremprincipal-1} $(i)$ is proved.
\section{Regularity and nearby cycles}
The aim of this section is to prove theorem \ref{comparaisonreg}. 
\subsection{} We will use the following
\begin{lemme}\label{lemmeevitement}
Let  $F$ be a germ of closed analytic subspace at the origin $0\in \mathds{C}^{n}$. Let $Y_1, \dots, Y_k$ be irreducible closed analytic subspaces of $\mathds{C}^{n}$ containing 0 and such that $F\cap Y_i$ is a strict closed subset of $Y_i$ for every  $i$. Then, there exists a germ of hypersurface $Z$ at the origin containing $F$ and such that $Z\cap Y_i$ has codimension 1 in $Y_i$ for every  $i$. 
\end{lemme}
\begin{proof}
Denote by $\mathcal{I}_{F}$ (resp. $\mathcal{I}_{Y_i}$) the ideal sheaf of $F$ (resp. $Y_i$). By irreducibility, $\mathcal{I}_{Y_i, 0}$ is a prime ideal in  $\mathcal{O}_{\mathds{C}^{n},0}$. The hypothesis say $\mathcal{I}_{F}\nsubseteq \mathcal{I}_{Y_i}$ for every $i$. From \cite[1.B]{Mat2}, we deduce
$$\mathcal{I}_{F}\nsubseteq \bigcup_i\mathcal{I}_{Y_i}$$
Any function $f\in \mathcal{I}_{F}$ not in  $\bigcup_i\mathcal{I}_{Y_i}$ defines a hypersurface as wanted.
\end{proof}
\subsection{}\label{lemmemeb}
We say that a holonomic module $\mathcal{M}$ is \textit{smooth} if the support $\Supp\mathcal{M}$ of $\mathcal{M}$ is smooth equidimensional and if the characteristic variety of $\mathcal{M}$ is equal to the conormal of $\Supp \mathcal{M}$ in $X$. We denote by $\Sing  \mathcal{M}$ the complement of the smooth locus of $\mathcal{M}$. It is a strict closed subset of $\Supp \mathcal{M}$.\\ \indent
Let $x\in X$ and let us define $F$ as the union of $\Sing \mathcal{M}$ with the irreducible components of $\Supp \mathcal{M}$ passing through $x$ which are not of maximal dimension. Define $Y_1, \dots, Y_k$ to be  the irreducible components of  $\Supp \mathcal{M}$ of maximal dimension passing through $x$.  From \ref{lemmeevitement}, one can find a hypersurface $Z$ passing through $x$ such that
\begin{enumerate}
\item  $Z\cap \Supp \mathcal{M}$ has codimension $1$ in $\Supp \mathcal{M}$.
\item The cohomology modules of $\mathcal{H}^{k}\mathcal{M}$ are smooth away from $Z$.
\item $\dim \Supp R\Gamma_{[Z]}\mathcal{M}<\dim \Supp \mathcal{M}$.
\end{enumerate}

\subsection{}\label{implicationdirecte} The direct implication of theorem \ref{comparaisonreg} is a consequence of the preservation of regularity by inverse image and the following
\begin{proposition}\label{premierinclusion}
We have $\mathcal{D}_{\hol}^b(X)_{\reg}\subset \mathcal{D}_{\hol}^b(X)_{\{0\}}$.
\end{proposition}
\begin{proof}
Take $\mathcal{M}\in\mathcal{D}_{\hol}^b(X)_{\reg}$. We argue by induction on  $\dim X$. The case where $X$ is a point is trivial. By arguing on $\dim \Supp \mathcal{M}$ as in  \ref{reduction}, we are left to prove \ref{premierinclusion} in the case where $\mathcal{M}$ is a regular meromorphic connection. Let $D$ be the pole locus of $\mathcal{M}$. Take $f\in \mathcal{O}_X$ and let $N$ with slope $>0$. To prove
$$
\psi_{f}(\mathcal{M}\otimes f^{+}N)\simeq 0
$$
one can suppose using embedded desingularization that $D+\divi f$ is a normal crossing divisor. We then conclude with \ref{lemmeannulation}.
\end{proof}
\subsection{} 
To prove the reverse implication of theorem \ref{comparaisonreg}, we argue by induction on  $\dim X\geq 1$. The case of curves follows from
\ref{propdim1}. We suppose that $\dim X\geq 2$ and we take
$\mathcal{M}\in \mathcal{D}_{\hol}^b(X)_{\{0\}}$. We argue by induction on $\dim \Supp \mathcal{M}$. The case where $\Supp \mathcal{M}$ is punctual is trivial.\\ \indent
Suppose that $0<\dim\Supp \mathcal{M}<\dim X$. Since $\Supp \mathcal{M}$ is a strict closed subset of $X$, one can always locally write $X=X^{\prime}\times D$ where $D$is the unit disc of  $\mathds{C}$ and where the projection  $X^{\prime}\times D\longrightarrow X^{\prime}$ is finite on $\Supp \mathcal{M}$. Let  $i:X^{\prime}\times D\longrightarrow X^{\prime}\times \mathds{P}^1$ be the canonical immersion. There is a commutative diagram
\begin{equation}\label{diagcom}
\xymatrix{
\Supp \mathcal{M} \ar[r] \ar[rd] &X^{\prime}\times \mathds{P}^1\ar[d]^-{p}\\
& X^{\prime}
}
\end{equation}
The oblique arrow of \eqref{diagcom} is finite, and $p$ is proper. So the horizontal arrow is proper. Thus, $\Supp \mathcal{M}$ is a closed subset in $X^{\prime}\times \mathds{P}^1$. Hence,  $\mathcal{M}$ can be extended by $0$ to $X^{\prime}\times \mathds{P}^1$. We still denote by $\mathcal{M}$ this extension. It is an object of $\mathcal{D}_{\hol}^b(X^{\prime}\times \mathds{P}^1)_{\{0\}}$ and we have to show that it is regular.\\ \indent
Let  $Z$ be a divisor in $X^{\prime}$ given by the equation $f=0$ and let $\rho: Y\longrightarrow X^{\prime}$ be a finite morphism. Since $p$ is smooth, the analytic space $Y^{\prime}$ making the following diagram 
$$
\xymatrix{
Y^{\prime} \ar[r]^-{\rho^{\prime}} \ar[d]_-{p^{\prime}} &X^{\prime}\times \mathds{P}^1\ar[d]^-{p}\\
Y \ar[r]_-{\rho} & X^{\prime}
}
$$
cartesian is smooth. Moreover $\rho^{\prime}$ is finite. By base change \cite[1.7.3]{HTT}, projection formula and compatibility of $\psi$ with proper direct image, we have for every germ $N$ of meromorphic connection with slope $>0$
\begin{align*}
\psi_{f}(\rho^{+}p_+\mathcal{M}\otimes f^{+}N)&\simeq \psi_{f}(p_+^{\prime}\rho^{\prime +}\mathcal{M}\otimes f^{+}N)\\
&\simeq 
\psi_{f}(p_+^{\prime}(\rho^{\prime +}\mathcal{M}\otimes (fp^{\prime})^{+}N))\\
&\simeq 
p_+^{\prime}\psi_{fp^{\prime}}(\rho^{\prime +}\mathcal{M}\otimes (fp^{\prime})^{+}N)\\
&\simeq 0
\end{align*}
By induction hypothesis $p_+\mathcal{M}$ is regular. Let $Y_1, \dots, Y_n$ be the irreducible components of $\Supp \mathcal{M}$ with maximal dimension. Since $\Sing  \mathcal{M}\cap Y_i$ is a strict closed subset of $Y_i$ and since a finite morphism preserves dimension, $p(\Sing  \mathcal{M})\cap p(Y_i)$ is a strict closed subset of the irreducible closed set $p(Y_i)$. In a neighbourhood of a given point of $p(\Sing  \mathcal{M})$, one can find from \ref{lemmemeb} a hypersurface $Z$ containing  $p(\Sing  \mathcal{M})$ such that $Z\cap p(Y_i)$ has codimension 1 in  $p(Y_i)$ for every $i$. So  $p^{-1}(Z)$ contains $\Sing  \mathcal{M}$ and 
$$\dim p^{-1}(Z)\cap Y_i=\dim Z \cap p(Y_i)=\dim p(Y_i)-1=\dim Y_i-1$$
Since $\Irr^{\ast}_{Z}$ is compatible with proper direct image \cite[3.6-6]{Mehbsmf}, we have
$$
\Irr^{\ast}_{Z}p_+\mathcal{M}\simeq Rp_{\ast}\Irr^{\ast}_{p^{-1}(Z)}\mathcal{M}\simeq 0
$$
Since $p$ is finite over $\Supp \mathcal{M}$, we have
$$
Rp_{\ast}\Irr^{\ast}_{p^{-1}(Z)}\mathcal{M} \simeq p_{\ast}\Irr^{\ast}_{p^{-1}(Z)}\mathcal{M}
$$
So for every $x\in p^{-1}(Z)$, the germ of  $\Irr^{\ast}_{p^{-1}(Z)}\mathcal{M}$ at $x$ is a direct factor of the complex $(p_{\ast}\Irr^{\ast}_{Z}p_+\mathcal{M})_{p(x)}\simeq 0$. Thus  $\Irr^{\ast}_{p^{-1}(Z)}\mathcal{M}\simeq 0$. From \cite[4.3-17]{Mehbsmf}, We deduce that $\mathcal{M}(\ast p^{-1}(Z))$ is regular. \\ \indent
To show that $\mathcal{M}$ is regular, we are left to prove that $R\Gamma_{[p^{-1}(Z)]}\mathcal{M}$ is regular. From \ref{implicationdirecte}, the nearby slopes of all quasi-finite inverse images of  $\mathcal{M}(\ast p^{-1}(Z))$ are contained in $\{0\}$. Thus, this is also the case for $R\Gamma_{[p^{-1}(Z)]}\mathcal{M}$. By construction of $Z$, $$\dim \Supp R\Gamma_{[p^{-1}(Z)]}\mathcal{M} < \dim \Supp  \mathcal{M}$$ 
We conclude by applying the induction hypothesis to  $R\Gamma_{[p^{-1}(Z)]}\mathcal{M}$. \\ \indent
Let us suppose that $\Supp \mathcal{M}$ has dimension $\dim X$, and let $Z$ be a hypersurface as in \ref{lemmemeb}. Then $\mathcal{M}(\ast Z)$ is a meromorphic connection with poles along $Z$. Let us show that  $\mathcal{M}(\ast Z)$ is regular. By \cite[4.3-17]{Mehbsmf}, it is enough to prove regularity generically along $Z$. Hence, one can suppose that $Z$ is smooth. By Malgrange theorem \cite{Reseaucan}, one can suppose that $Z$ is smooth and that $\mathcal{M}(\ast Z)$ has good formal structure along $Z$. Let $(x_1,\dots, x_n, t)$ be coordinates centred at $0\in Z$ such that $Z$ is given by $t=0$ and let 
 $\rho: (x,u)\longrightarrow (x,u^{p})$ be as in \ref{bonnedecompositionformelle} for $\mathcal{M}(\ast Z)$. Let
$ \mathcal{E}^{g(x,u)/u^{k}}\otimes \mathcal{R}$ be a factor of  $\rho^{+}(\hat{\mathcal{M}}_0(\ast Z))$
where $g(0,0)\neq 0$ and where $\mathcal{R}$ is a regular meromorphic connection with poles along $Z$. For a choice of $k$-th root in a neighbourhood of $g(0,0)$, we have
$$
\psi_{u/\sqrt[k]{g}}(\rho^{+}\mathcal{M}\otimes (u/\sqrt[k]{g})^{+}\mathcal{E}^{-1/u^{k}})\simeq 0
$$
Since nearby cycles commute with formalization, we deduce
$$
\psi_{u}(\rho^{+}(\hat{\mathcal{M}}_0(\ast Z))\otimes \mathcal{E}^{-g/u^{k}})\simeq \psi_{u}(\rho^{+}\hat{\mathcal{M}}_0\otimes \mathcal{E}^{-g/u^{k}})\simeq 0
$$
Thus $\psi_u \mathcal{R}\simeq 0$, so $\mathcal{R}\simeq 0$. Hence, the only possibly non zero factor of $\rho^{+}(\hat{\mathcal{M}}_0(\ast Z))$ is the regular factor. So $\mathcal{M}(\ast Z)$ is regular. We obtain that $\mathcal{M}$ is regular by applying the induction hypothesis to $R\Gamma_{[Z]}\mathcal{M}$.

\section{Slopes and irregular periods}

\subsection{}\label{DRrapide}
Let $X$ be a smooth complex manifold and let $D$ be a normal crossing divisor in $X$. Define $U:= X\setminus D$ and let $j:U \longrightarrow X$ be the canonical inclusion. Let $\mathcal{M}$ be a meromorphic connexoin on  $X$ with poles along $D$.\\ \indent
We denote by $p:\tilde{X}\longrightarrow X$ the real blow-up of $X$ along $D$. Let $\mathcal{A}_{\tilde{X}}^{<D}$ be the sheaf \cite[II]{Sabbahdim} of differentiable functions on $\tilde{X}$ whose restriction to $X$ are holomorphic and whose asymptotic development along $p^{-1}(D)$ are zero. We define the  \textit{de Rham complex with rapid decay} by
$$
\DR^{<D}_{\tilde{X}}\mathcal{M}:=\mathcal{A}_{\tilde{X}}^{<D}\otimes_{p^{-1}\mathcal{O}_{X}}p^{-1}\DR_{X}\mathcal{M}
$$
\subsection{}
With the notations in \ref{DRrapide}, if $\mathcal{M}$ has good formal structure along $D$, we define \cite[Prop 2]{HienInv}
$$
H^{\rd}_k(X,\mathcal{M}):=H^{2d-k}(\tilde{X},\DR^{<D}_{\tilde{X}}\mathcal{M})
$$
The left-hand side is the space of \textit{cycles with rapid decay for $\mathcal{M}$}. For a topological description justifying the terminology, we refer to \cite[5.1]{HienInv}.
\subsection{Proof of theorem \ref{GMirr}}
We denote by  $j:U\longrightarrow X$ the canonical immersion, $d:=\dim X$ and $\Sl_0(\mathcal{H}^{k}f_+\mathcal{E})$ the slopes of $\mathcal{H}^{k}f_+\mathcal{E}$ at 0. We will also use the letter $f$ for the restriction of $f$ to $U$. From \cite[4.7.2]{HTT}, we have a canonical identification
\begin{equation}\label{isoan}
\xymatrix{
(f_+\mathcal{E})^{\an}\simeq  (f_{+}(j_+\mathcal{E}))^{\an}\ar[r]^-{\sim} & f_{+}^{\an}(j_+\mathcal{E})^{\an}
}
\end{equation}
We deduce 
$$
\Sl_0(\mathcal{H}^{k}f_+\mathcal{E})=\Sl_0(\mathcal{H}^{k}f_{+}^{\an}(j_+\mathcal{E})^{\an})
$$
Let $x$ be a local coordinate on $S$ centred at the origin. From \ref{propdim1}, we have 
$$
\Sl_0(\mathcal{H}^{k}f_{+}^{\an}(j_+\mathcal{E})^{\an})=\Sl_x^{\nb}(\mathcal{H}^{k}f_{+}^{\an}(j_+\mathcal{E})^{\an})
$$
Since $\Sl_x^{\nb}(\mathcal{H}^{k}f_{+}^{\an}(j_+\mathcal{E})^{\an})\subset \Sl_x^{\nb}(f_{+}^{\an}(j_+\mathcal{E})^{\an})$, we deduce from theorem \ref{theoremprincipalraffiné} and theorem \ref{theoremprincipal-1} 
$$
\Sl_0(\mathcal{H}^{k}f_+\mathcal{E})\subset \Sl_{f(x)}^{\nb}((j_+\mathcal{E})^{\an})\subset [0,r_1+\cdots +r_n]
$$
We are thus left to relate $\Sol(\mathcal{H}^{k}f_{+}^{\an}(j_+\mathcal{E})^{\an})$ to the periods of $\mathcal{E}_{t}$, for $t\neq 0$ close enough to $0$. Such a relation appears for a special type of rank 1 connections in \cite{HR}. We prove more generally the following
\begin{proposition}\label{derniereprop}
For every $k$, we have a canonical isomorphism
\begin{equation}\label{isocannonique}
\xymatrix{
R^{k}f_{\ast}^{\an}\Sol (j_+\mathcal{E})^{\an}\ar[r]^-{\sim}& R^{k}(f^{\an}p)_{\ast}\DR^{<D}_{\tilde{X}}(j_+\mathcal{E}^{\ast})^{\an}
}
\end{equation}
For $t\neq 0$ close enough to 0, the fiber of the right-hand side of \eqref{isocannonique} at $t$ is canonically isomorphic to $H_{2d-2-k}^{\rd}(U_t, \mathcal{E}^{\ast}_t)$. 
\end{proposition}
\begin{proof}
Set
$
\mathcal{M}:=(j_+\mathcal{E})^{\an}
$. Since $\mathcal{M}$ has good formal structure, we know from
\cite[5.3.1]{MemoireMochizuki} that
$$
Rp_{\ast}\DR^{<D}_{\tilde{X}}\mathcal{M}\simeq \DR (\mathcal{M}(!D))
$$
where $\mathcal{M}(!D):=\mathds{D}((\mathds{D}\mathcal{M})(\ast D))$. We have $(\mathds{D}\mathcal{M})(\ast D)=\mathcal{M}^{\ast}$ where
$\mathcal{M}^{\ast}$  is the dual connection of $\mathcal{M}$. Thus we have
$$
\Sol\mathcal{M}\simeq \DR\mathds{D}\mathcal{M}\simeq Rp_{\ast}\DR^{<D}_{\tilde{X}}\mathcal{M}^{\ast}
$$
By applying $Rf_{\ast}^{\an}$, we obtain for every  $k$ and every $t\neq 0$ close enough to 0
$$
\xymatrix{
(R^{k}f_{\ast}^{\an}\Sol \mathcal{M})_{t}\ar[d]_-{(1)} \ar[r]^-{\sim}& (R^{k}(f^{\an}p)_{\ast}\DR^{<D}_{\tilde{X}}\mathcal{M}^{\ast})_t     \ar[d]^{(6)} 
 \\
      H^{k}(X_t^{\an},(\Sol\mathcal{M})_t)     \ar[d]_-{(2)}& 
H^{k}(X_t^{\an}, (\DR^{<D}_{\tilde{X}}\mathcal{M}^{\ast})_t)     \ar[dd]^-{(7)}  
           \\
    H^{k}(X_t^{\an},\Sol\mathcal{M}_t) \ar[d]_-{(3)}   &
\\
H^{-k}(X_t^{\an}, \mathds{D}\Sol\mathcal{M}_t)^{\ast} \ar[d]_-{(4)} &      H^{k}(X_t^{\an}, \DR^{<D_t}_{\tilde{X_t}}\mathcal{M}^{\ast}_t)    \ar[d]^-{\wr}  \\
H^{2d-2-k}(X_t^{\an}, \DR\mathcal{M}_t)^{\ast} \ar[d]_-{(5)} &
H_{2d-2-k}^{\rd}(X_t^{\an}, \mathcal{M}_t^{\ast})\ar[d]^-{|}\\
  H^{2d-2-k}(U_t, \DR\mathcal{E}_t)^{\ast} \ar[r]_-{(8)}&H_{2d-2-k}^{\rd}(U_t, \mathcal{E}^{\ast}_t)
}   
$$
By proper base change theorem, the morphisms $(1)$ and $(6)$ are isomorphisms. The morphism $(2)$ is an isomorphism by non charactericity hypothesis. The morphism $(3)$ is an isomorphism by Poincaré-Verdier duality. The morphism $(4)$ is an isomorphism by duality theorem for $\mathcal{D}$-modules \cite{TheseMeb}\cite{KashKawai}. The morphism $(5)$ is an isomorphism by GAGA and exactness of $j_{t\ast}$ where $j_t:U_t\longrightarrow X_t$ is the inclusion morphism. The morphism $(8)$ is an isomorphism by Hien duality theorem. We deduce that $(7)$ is an isomorphism. 
\end{proof}
Let  $\mathbf{e}:=(e_1, \dots, e_n)$ be a local trivialization of $\mathcal{H}^{k}(f_{+}\mathcal{E})(\ast 0)$ in a neighbourhood of 0. One can suppose that $f$ is smooth above $S^{\ast}:=S\setminus \{ 0\}$. Set $U^{\ast}:=U\setminus \{f^{-1}(0)\}$. From                       \cite[1.4]{DiMaSaSai}, we have an isomorphism of left $\mathcal{D}_S$-modules
$$
\mathcal{H}^{k}(f_{+}\mathcal{E})_{|S^{\ast }}\simeq 
 R^{k+d-1}f_{\ast}\DR_{U^{\ast}/S^{\ast}}\mathcal{E}
$$
where the right hand side is endowed with the Gauss-Manin connection as defined in \cite{KatzOda}. We deduce that  $(\mathbf{e}_t)_{t\in S^{\ast}}$ is an algebraic family of bases for the  $H^{k+d-1}_{\dR}(X_t, \mathcal{E}_t)$. \\ \indent
At the cost of shrinking $S$, Kashiwara perversity theorem \cite{Ka2} shows that the only possibly non zero terms of the hypercohomology spectral sequence 
$$
E_{2}^{pq}=\mathcal{H}^{p}\Sol \mathcal{H}^{-q}(f_+\mathcal{E})^{\an}_{|S^{\ast }}  \Longrightarrow \mathcal{H}^{p+q}\Sol (f_+\mathcal{E})^{\an}_{|S^{\ast }}
$$
sit on the line $p=0$. Hence, at the cost of shrinking $S$ again, we have 
\begin{equation}\label{derniereequation}
\Sol \mathcal{H}^{k}(f_+\mathcal{E})^{\an}_{|S^{\ast }}\simeq  \mathcal{H}^{0}\Sol \mathcal{H}^{k}(f_+\mathcal{E})^{\an}_{|S^{\ast }}\simeq \mathcal{H}^{-k}\Sol (f_+\mathcal{E})^{\an}_{|S^{\ast }}
\end{equation}
Since $\Sol$ is compatible with proper direct image, we deduce from \eqref{isoan} and \eqref{derniereequation}
\begin{equation}\label{derniereequationbis}
\Sol \mathcal{H}^{k}(f_+\mathcal{E})^{\an}_{|S^{\ast }}\simeq R^{-k+d-1}f_\ast\Sol (j_{+}\mathcal{E})^{\an} 
\end{equation}
Let $\mathbf{s}:=(s_1, \dots, s_n)$ be a local trivialization of  $\Sol \mathcal{H}^{k}(f_{+}\mathcal{E})^{\an}$ over an open subset in $S^{\ast \an}$. From \eqref{derniereequationbis} and \ref{derniereprop},  we deduce that  $(\mathbf{s}_t)_{t\in S^{\ast}}$ is a continuous family of basis for the spaces $H_{k+d-1}^{\rd}(U_t, \mathcal{E}^{\ast}_t)$. The periods of $\mathcal{E}$ are by definition the coefficients of $\mathbf{s}$  in $\mathbf{e}$, and theorem \ref{GMirr} is proved.

\bibliographystyle{amsalpha}
\bibliography{Saito-Sato-Seminar}

\end{document}